\documentclass[preprint,12pt]{elsarticle}
\usepackage{amsmath,amsthm,amsfonts,amssymb}
\usepackage{mathtools}
 \newcommand{\R}{{\mathbb R}}
  
 \usepackage[active,new,old]{correct} 

\usepackage{graphicx,latexsym}
 \usepackage{hyperref}
\usepackage{lscape, marvosym}

\newtheorem{thm}{Theorem}[section]
 \newtheorem{cor}[thm]{Corollary}
 \newtheorem{lem}[thm]{Lemma}
 
\newtheorem{ex}[thm]{Example}
 \newtheorem{dfn}[thm]{Definition}

\newdefinition{remark}{Remark}

\newtheorem{discu}{Discussion:}

\numberwithin{equation}{section}
\newtheorem{conje}{Conjecture:}

\begin{document}
\vspace{-.4cm}
\begin{frontmatter}
\vspace{-.4cm}

\title{\sc{Comparison results for proper
multisplittings  of  rectangular matrices}}

\vspace{-.4cm}

\author{Chinmay Kumar Giri$^\dag$$^a$ and Debasisha Mishra$^\dag$$^b$}
\vspace{-.2cm}
\address{

                       $^{\dag}$ Department of Mathematics,\\
                        National Institute of Technology Raipur, India.
                        \\email$^a$:  ckg2357\symbol{'100}gmail.com
                        \\email$^b$:  dmishra\symbol{'100}nitrr.ac.in. \vspace{-.6cm}}

\vspace{-.62cm}
\begin{abstract}
The least square solution of minimum norm of a rectangular linear
system of equations can be found out iteratively by using  matrix
splittings. However, the convergence of such an iteration scheme
arising out of a matrix splitting is practically very slow in many
cases. Thus, works on improving the speed of the iteration scheme
have attracted great interest. In this direction, comparison of the
rate of convergence of the iteration schemes produced by two matrix
splittings is very useful. But, in the case of matrices having many
matrix splittings, this process is time-consuming. The main goal of
the current article is to provide a solution  to the above issue by
using proper multisplittings.  To this end, we propose a few
comparison theorems for proper weak regular splittings and proper
nonnegative splittings first. We then derive convergence and
comparison theorems for proper multisplittings with the  help of the
theory of proper weak regular splittings.
\end{abstract}

\vspace{-.2cm}

\begin{keyword}
Non-negativity; Moore-Penrose inverse; Proper splitting;
Multisplittings; Convergence theorem; Comparison theorem.

\vspace{.1cm} { \it AMS subject classifications:}  15A09
\end{keyword}
\vspace{-.2cm}
\end{frontmatter}
\vspace{-.2cm}
\newpage

\section{Introduction}
Let us consider a rectangular system of linear equations of the form
\begin{equation}\label{introeq1}
Ax=b,
\end{equation}
where $A$ is a real, large and sparse matrix of order $m \times n,$
$x$ is an unknown
 real $n$-vector, and $b$ is a given real $m$-vector. If (\ref{introeq1}) is inconsistent,
 then one usually seeks the least square solution of minimum norm. This solution vector $x$ is
 then computed  by $x=A^{\dagger}b$, where $A^{\dagger}$ is the Moore-Penrose  inverse
   of $A$ (see Section 2, for its definition). In a
wide variety of such problems, including the Neumann problem and
those for elastic bodies with free surfaces, the finite difference
formulations lead to a singular, consistent linear system of the
form (\ref{introeq1}), where A is large and sparse. In these
situations, one can opt for an iterative  method for finding  the
least square solution of minimum norm. Such a method where A is
rectangular or (\ref{introeq1}) is inconsistent, is studied in
\cite{bpcones}. In particular, the authors of \cite{bpcones} have
introduced the following
  iteration scheme to find the least square solution of minimum norm of the system (\ref{introeq1})
\begin{equation}\label{introeq2}
x^{i+1}=U^{\dagger}Vx^{i}+U^{\dagger}b, \quad i=0,1,2,\ldots,
\end{equation}
where $A=U-V$ is a proper splitting. A splitting\footnote{A
 splitting of a real rectangular matrix $A$ is an expression of the form
 $A = U-V$, where $U$ and $V$ are matrices of the same order as in
 $A$.} $A=U-V$ of $A\in
{\R}^{m\times n}$ (the set of all real $m\times n$ matrices) is
called a {\it proper splitting} (\cite{bpcones}) if $R(U)=R(A)$ and
$N(U)=N(A)$, where $R(A)$ and $N(A)$  denote  the range space and
the null space of $A$, respectively. The iteration scheme
(\ref{introeq2}) is said to be {\it convergent} if the spectral
radius of $U^{\dag}V$ is less than 1, and $U^{\dag}V$ is called the
{\it iteration matrix}. For the proper splitting $A=U-V$, the same
authors (\cite{bpcones}) have shown that the iteration scheme
(\ref{introeq2}) converges to $x=A^{\dag}b$, the least squares
solution of minimum norm, for any initial vector $x^0$ if and only
if  the iteration scheme (\ref{introeq2}) is convergent (see
Corollary 1, \cite{bpcones}). The advantage of the iterative method
for solving the rectangular system of linear equations
(\ref{introeq1}) is that it avoids the use of the normal system
$A^TAx = A^Tb$, where $A^TA$ is frequently ill-conditioned and
influenced greatly by roundoff errors (see \cite{f}). (Here $A^T$
stands for the transpose of a matrix $A$.)

 Berman and Plemmons \cite{bpcones}  have proved a
few convergence results for different classes of proper splittings
without calling them by any name. Later on, Climent and Perea
\cite{c3}, Climent {\it  et al.} \cite{c1} have introduced different
classes of proper splittings and studied its convergence theory.
Subsequently, it is  carried forward by Mishra and Sivakumar
\cite{misoam},  Jena {\it et al.} \cite{miscal}, Mishra
\cite{miscma},  Baliarsingh and Mishra \cite{alekha2}, and Giri and
Mishra \cite{d3}, to name a few.  Here we list  three
 important classes of proper splittings.
A proper splitting $A=U-V$ of $A\in {\R}^{m\times n}$   is called a \\
  (i)  {\it proper regular splitting} if   $U^{\dag}\geq 0$ and $V\geq 0$ (\cite{miscal}),\\
  (ii)  {\it proper  weak regular splitting} if  $U^{\dag}\geq 0$ and $U^{\dag}V\geq 0$ (\cite{miscal}),\\
  (iii)  {\it proper nonnegative splitting}  if $U^{\dag}V\geq 0$
  (\cite{miscma}).\\
($B\geq 0$ means each entry of $B$ is non-negative, and more on
these classes of proper splittings will be discussed in further
sections.) In the case of nonsingular matrices, the above
definitions coincide with regular (\cite{var}), weak regular
(\cite{var}) and nonnegative (\cite{songlaa}) splitting,\footnote{A
splitting $A=U-V$ of a real square matrix $A$ is called regular
(\cite{var}), if $U^{-1}$ exists, $U^{-1} \geq 0$
 and $V\geq 0$, weak regular (\cite{var}) if $U^{-1}$ exists, $U^{-1} \geq 0$
 and $U^{-1}V\geq 0$, and  nonnegative (\cite{songlaa}) if $U^{-1}$ exists and $U^{-1}V\geq 0$, respectively.}
 respectively.

Comparison theorems between the spectral radii of matrices are
useful tools in the analysis of  the rate of convergence of
iterative methods or for judging the efficiency of preconditioners.
A matrix
 $A$ may have different matrix splittings (say $A=U_1-V_1=U_2-V_2$). In
practice, we seek such an $U$ which not only makes the computation
$x^{i+1}$(given $x^{i}$) simpler but also yields the spectral radius
of $U^{\dag}V$ (which is of course less than 1)  as small as
possible for the faster  rate of convergence of the iteration scheme
(\ref{introeq2}). An accepted rule for preferring one iteration
scheme to another is to choose the iteration scheme  having the
smaller spectral radius. In this context, Jena {\it et al.}
\cite{miscal}, Mishra and Sivakumar \cite{misaml}, Mishra
\cite{miscma} and  Baliarsingh and Mishra  \cite{alekha2} have proved various comparison results for different
class of matrix splittings of rectangular matrices. In this article,
we propose a few more comparison results.

But one of the drawbacks of the above-discussed theory is that this
process needs more time when a matrix has many splittings as one can
compare two matrix splittings at a time. A natural question arises
at this level is ``can we have a faster iteration scheme than
(\ref{introeq2})''. This is answered by O'Leary and White
\cite{white} who have introduced the
  concept of the multisplitting method for obtaining the parallel solution of linear system of equations
   of the form (\ref{introeq1}), but in  the square nonsingular matrix setting.
A real $n \times n$ matrix $A$ is called {\it monotone} (or a matrix
of ``{\it monotone kind}")
 if $Ax \geq  0 ~\Rightarrow x \geq 0$. Here, $y\geq  0$ for $(y_1,y_2,\cdots,y_n)^{T}=y \in \mathbb{R}^n$
 means that $y_i \geq 0$ (or $y_i$ is non-negative) for all $i=1,2,\ldots,n$. This notion was introduced by
 Collatz,
 who has shown that $A$ is monotone if and only if $A^{-1}$ exists and $A^{-1} \geq 0$.
  The book by Collatz \cite{col2} has details of how monotone
   matrices arise naturally in the study of finite difference approximation methods for
    certain elliptic partial differential equations. The problem of characterizing monotone
    (also referred  as {\it inverse positive}) matrices in terms of matrix splittings
    has been extensively dealt with in the
    literature.   The books by Berman and Plemmons
\cite{bpbook} and Varga \cite{var} give an excellent account
     of many of these characterizations and its extension to
     rectangular matrices. O'Leary and White \cite{white} have provided
    the convergence theory of multisplittings for the class of monotone
     matrices, and is explained below.

   The triplet
   $(U_{k},V_{k},E_{k})_{k=1}^{p}$ is called a {\it multisplitting} of $A\in {\R}^{n\times n}$ if\\
$(i)~A=U_{k}-V_{k},~ \text{for each}~ k=1,2, \ldots,p$,\\
$(ii)~E_{k} \geq 0$ is a non-zero and diagonal matrix, for each $k=1,2,\ldots,p,$\\
$(iii)\displaystyle\sum_{k=1}^{p}E_{k}=I,~\text{where $I$ is the
identity matrix.}$

Using  the multisplitting $(U_{k},V_{k},E_{k})_{k=1}^{p}$, the
authors of \cite{white} have considered the following iteration
scheme:
\begin{equation}\label{introeq3}
x^{i+1}=Hx^{i}+Gb,\quad i=0,1,2,\ldots,
\end{equation}
where $\displaystyle H=\sum_{k=1}^{p}E_{k}U_{k}^{-1}V_{k}~~
\text{and}~~G=\sum_{k=1}^{p}E_{k}U^{-1}_{k}.$ The same authors
\cite{white} have shown that if $A=U_{k}-V_{k},\quad k=1,2,\ldots,p$
is a weak regular splitting
 of a monotone matrix $A$, then the iteration scheme (\ref{introeq3}) converges for any
  initial vector $x^{0}$. 

In contrast to the vast literature available on solving the square
nonsingular system of linear equations, iteratively, the researches
on solving the rectangular system of linear equations, iteratively
are limited. In particular, the theory of multisplittings has not
been studied much for rectangular matrices. Climent and Perea
\cite{c3} first introduced the concept of a proper multisplitting.
Thereafter, Baliarsingh and Jena \cite{alekha1} applied the same
theory to solve the square singular system of linear equations.  In
this note, we revisit the same theory first and  add a few more
results to existing theory with the objective to solve the
rectangular linear systems. Some of the results obtained in this
paper dealing with multisplittings theory are completely new even
for square nonsingular matrices.

The contents of this paper are organized in the following order.
Next Section includes some notation and fundamental concepts
concerned in our study. In Section 3 we set up the background, and
then establish a number of comparison results between two proper
weak regular splittings
 of different types. This is a prelude to Section 4 in which we
study  similar results as of section 3, but for proper nonnegative
splittings
 of different types.   Section 5 is devoted to the study of multisplittings of a rectangular matrix.
  Finally, Section 6 gives the conclusions of this
work.

\section{Preliminaries}
To present a remarkably reader-friendly convergence analysis of
rectangular matrix splittings, we first explain some basic notation
and definitions. In the subsequent sections, ${\R}^n$ means an
n-dimensional Euclidean space. If $L\oplus M= {\R}^n$, then
$P_{L,M}$ is referred as the projection onto $L$ along $M$. So,
$P_{L,M}A=A$ if and only if $R(A)\subseteq L$ and $AP_{L,M}=A$ if
and only if $N(A) \supseteq M$. If $L\perp M$, then $P_{L,M}$ will
be denoted by $P_{L}$.
%
 For $A\in {\R}^{m \times n}$, the
unique matrix $X\in{\R}^{n \times m}$ is called the
\textit{Moore-Penrose inverse} of $A$ if it satisfies the following
four equations:
 $$AXA=A,\quad XAX=X,\quad (AX)^{T}=AX ~~\text{and}~~(XA)^{T}=XA,$$
 and is denoted by $A^{\dagger}$. It always exists,
   and $~A^{\dagger}=A^{-1}$ in the case of a nonsingular matrix $A$.  Properties of
$A^{\dagger}$  which will be frequently used in this paper are:
 $R(A^{\dagger})=R(A^{T});~N(A^{\dagger})=N(A^{T});~AA^{\dagger}=P_{R(A)}$ and $A^{\dagger}A=P_{R(A^{T})}$
 (see \cite{ben} for more details).
%
%
A matrix $A \in {\R}^{m \times n}$ is called \textit{non-negative}
if $A\geq 0$, and $B \geq C$ if $B-C \geq 0$. Again, $B \gneq C$ means $B \geq C$ and $B \neq C$. Similarly,
  a matrix $A\in {\R}^{m \times n}$ is called \textit{positive} if each element of $A$ is positive,
  and is denoted by $A>0$. We also use the above notation for vectors as vectors can be seen as $n\times 1 $ matrices.
   A matrix $A \in {\R}^{m \times n}$ is called \textit{semimonotone} if
   $A^{\dagger}\geq0$.  For a matrix $A=(a_{ij}) \in {\R}^{n \times n}$,  the  set  of indices
    $i,j=  1, 2,\ldots, n$  will  be  denoted  by  $S$. A matrix $A$ is \textit{reducible} if there
     exists a nonvoid index set
$R$, $R \subset S$ and $R\neq S$ such that $a_{ij} = 0$ for $i\in R$
and $j \in S - R$, otherwise the matrix $A$ is \textit{irreducible}.
Clearly, each positive matrix is irreducible.
   The \textit{spectral radius} of a matrix $A \in {\R}^{n \times n}$ is denoted by $\rho(A)$, and is equal
   to the maximum of
     the moduli of the eigenvalues of $A$. Let $A$ and $B$ be two
      matrices of appropriate order such that the products
      $AB$ and $BA$ are defined. Then $\rho(AB)=\rho(BA)$.
       Before proceeding further, we collect certain results which are going
to be used in the sequel.

\begin{thm}\label{perron}\textnormal{(Theorem 2.20, \cite{var})}\\
If $A \in {\R}^{n \times n}$ and $A \geq 0$, then\\
$(a)$ A has a non-negative real eigenvalue equal to its spectral radius.\\
$(b)$ To $\rho(A)\geq 0$, there corresponds an eigenvector $x\geq
0.$
\end{thm}

\begin{thm}\label{perron1}\textnormal{(Theorem 2.7, \cite{var})}\\
If $A \in {\R}^{n \times n}$ is an irreducible matrix and $A \geq 0$, then\\
$(a)$ $A$ has a positive real eigenvalue equal to its spectral radius.\\
$(b)$ To $\rho(A)$, there corresponds an eigenvector $x > 0.$\\
$(c)$ $\rho(A)$ increases when any entry of $A$ increases.
\end{thm}

\begin{thm}\textnormal{(Theorem 2.21,~\cite{var})}\label{var1}\\
If $A,~B \in {\R}^{n \times n}$ and $A \geq B \geq 0$, then
$\rho(A)\geq \rho(B).$
\end{thm}

\begin{thm}\textnormal{(Theorem 3.15, \cite{var})}\label{var2}\\
Let $X\in{\R}^{n \times n}$ and $X \geq 0$. Then $\rho(X)<1$ if and
only if $(I-X)^{-1}$
 exists and $(I-X)^{-1}=\displaystyle\sum_{k=0}^{\infty} X^{k}\geq 0$.
\end{thm}

\begin{lem}\label{lemma3.2}\textnormal{(Theorem 2.1.11, \cite{bpbook})}\\
Let $B \in {\R}^{n \times n},~B\geq 0,~x\geq 0 ~(x\neq 0)$ and $\alpha$ be a positive scalar.\\
$(i)~$ If $\alpha x \leq Bx $, then $\alpha \leq \rho(B)$. Moreover,
if $Bx> \alpha x$,
 then $\rho(B)> \alpha$.\\
$(ii)$ If $Bx \leq \alpha x,~x>0$, then $\rho(B) \leq \alpha$.
\end{lem}

\begin{lem} \textnormal{(Lemma 3.16, \cite{d4})}\label{lemmad4}\\
Let $A, B \in {\R}^{n \times n}$ be two semimonotone matrices such
that $R(A)=R(B)$ and $N(A)=N(B)$. If $A \geq B$, then $B^{\dagger}
\geq A^{\dagger}$.
\end{lem}
%
%
The following is the first result on a proper splitting.
\begin{thm}\textnormal{(Theorem 1, \cite{bpcones})}\label{bpthm}\\
Let $A=U-V$ be a proper splitting of $A\in{\R}^{m \times n}$. Then\\
$(a)~A=U(I-U^{\dagger}V),$\\
$(b)~I-U^{\dagger}V$ is nonsingular,\\
$(c)~A^{\dagger}=(I-U^{\dagger}V)^{-1}U^{\dagger}$.
\end{thm}

 Climent \textit{et al.} \cite{c1} again obtained a few more properties of
  a proper splitting which are reproduced next.

 \begin{thm}\textnormal{(Theorem 1, \cite{c1})}\label{cli}\\
Let $A=U-V$ be a proper splitting of $A\in {\R}^{m \times n}$. Then\\
$(a)~A=(I-VU^{\dagger})U$,\\
$(b)~A^{\dagger}=U^{\dagger}(I-VU^{\dagger})^{-1}$.
\end{thm}

The next lemma shows a relation between the eigenvalues of
$U^{\dagger}V$ and $A^{\dagger}V$.

\begin{lem}\label{lemqeigval}(Lemma 2.7,  \cite{d4})\\
Let  $A=U-V$ be a  proper splitting of $A\in {\R}^{m\times n}$. Let
$\mu_i, ~1\leq i \leq n$ and $\lambda_j, ~1\leq j \leq n$ be the
eigenvalues of the matrices $U^{\dag}V$ and $A^{\dag}V$,
respectively. Then, for every $j$, there exists $i$ such that
$\lambda_j =\frac{\mu_i}{1-\mu_i}$, and for every $i$, there exists
$j$ such that $\mu_i=\frac{\lambda_j}{1+\lambda_j}.$
\end{lem}

\section{Proper weak regular splittings of different types}
To make the article fairly self-contained, we shall briefly evoke
the notion of proper weak regular splittings of different types of
rectangular matrices and associated concepts in this section. To
prepare the setting,  we first need the following definition.

\begin{dfn} \textnormal{(Definition 1.1, \cite{miscal})}\\
 A proper splitting $A=U-V$  of $A \in {\R}^{m \times n}$ is called a proper
  regular splitting if  $U^{\dagger} \geq 0$ and $V \geq 0$.
\end{dfn}

Jena {\it et al.}  \cite{miscal} proved the following comparison
theorem for proper regular splittings in order to improve
convergence speed of the iteration scheme (\ref{introeq2}).

\begin{thm}\label{calcolo}\textnormal{(Theorem 3.3, \cite{miscal})}\\
Let $A=U_{1}-V_{1}=U_{2}-V_{2}$ be two proper regular splittings of
a semimonotone matrix $A \in {\R}^{m \times n}$. If $U_{1}^{\dagger}
\geq U_{2}^{\dagger}$, then $$\rho(U_{1}^{\dagger}V_{1}) \leq
\rho(U_{2}^{\dagger}V_{2})<1.$$
\end{thm}

We next reproduce  the definition of  a larger class of matrices
than the class of proper
  regular splittings.

\begin{dfn}\textnormal{(Definition 1.2, \cite{miscal})}\\
 A proper splitting $A=U-V$  of $A \in {\R}^{m \times n}$ is called a proper weak regular splitting
 if  $U^{\dagger} \geq 0$ and $U^{\dagger}V \geq 0$.
\end{dfn}

The statement mentioned before the above Definition is shown below
with an example.

\begin{ex}
Let $A=\left(\begin{array}{ccc}
2 &-1& 2\\
-3 & 5 & -3
\end{array}
\right)=\left(\begin{array}{ccc}
2 &-2& 2\\
-3 & 10 & -3
\end{array}
\right)-\left(\begin{array}{ccc}
0 &-1& 0\\
0 & 5 & 0
\end{array}
\right)=U-V$. Then $R(U)=R(A)$, $N(U)=N(A)$,
$U^{\dagger}=\left(\begin{array}{ccc}
0.3571 & 0.0714 \\
0.2143 & 0.1429 \\
0.3571 & 0.0714
\end{array}
\right) \geq 0$ and $U^{\dagger}V=\left(\begin{array}{ccc}
0 & 0 & 0\\
0 & 0.5 & 0\\
0 & 0 & 0
\end{array}
\right) \geq 0$. Thus $A=U-V$
 is a proper weak regular splitting, but not a proper regular splitting.
  This is due to the fact that $V \ngeq 0$.
\end{ex}

Berman and Plemmons \cite{bpcones} obtained the following
convergence result for a proper weak regular splitting without
specifying the name of this class.

\begin{thm}\textnormal{(Corollary 4, \cite{bpcones})}\label{m1}\\
Let $A=U-V$ be a proper weak regular splitting of $A\in {\R}^{m
\times n}$. Then $A^{\dagger} \geq 0$ if and only if
$\rho(U^{\dagger}V)<1$.
\end{thm}

The next comparison result  is proved by Mishra \cite{d4}, and will
be used in Section 5.

\begin{thm}\textnormal{(Theorem 3.4, \cite{d4})}\label{d4}\\
Let $A=U_{1}-V_{1}=U_{2}-V_{2}$ be two proper weak regular
splittings of a semimonotone matrix
$A\in{\R}^{m \times n}$. If either of the following cases holds,\\
$(i)~V_{2} \geq V_{1}$\\
$(ii)~U_{1}^{\dagger} \geq U_{2}^{\dagger},~V_{1} \geq 0$\\
$(iii)~U_{1}^{\dagger} \geq U_{2}^{\dagger}\geq 0$ and
 row sums of $U_{2}^{\dagger}$ are positive, $V_{2} \geq 0,$\\
then $$\rho(U_{1}^{\dagger}V_{1}) \leq
\rho(U_{2}^{\dagger}V_{2})<1.$$
\end{thm}

One can find that, there exists a convergent splitting which is not
a  proper weak regular splitting. To address convergence theory in
this situation, we now have the following definition from \cite{c1}
where the authors call it as a weak nonnegative splitting of second
type. However, we call here as  a proper weak regular
 splitting of type II.

\begin{dfn}\textnormal{(Definition 2, \cite{c1})}\\
 A proper splitting $A=U-V$  of $A \in {\R}^{m \times n}$ is called a proper weak regular splitting of type II
  if  $U^{\dagger} \geq 0$ and $VU^{\dagger} \geq 0$.
\end{dfn}

Note that the proper weak regular splitting of type I is same as the
proper weak regular splitting.
 We next present an example of a matrix which has a convergent proper weak regular splitting
 of type II but not of type I.

\begin{ex}
Let $A=\left( \begin{array}{ccc}
3& -3 & 6\\
3 & 6 & -3\\
\end{array}
\right)=\left(
\begin{array}{ccc}
5 & -5 & 10\\
4 & 8 & -4\\
\end{array}
\right) -\left( \begin{array}{ccc}
2 & -2 & 4\\
1 & 2 & -1\\
\end{array}
\right)=U-V.$
 Then $R(U)=R(A),~N(U)=N(A),~ U^{\dagger}=\left( \begin{array}{ccc}
  0.0667 &  0.0833\\
   0 &   0.0833\\
    0.0667  &  0
 \end{array} \right) \geq 0$ and $VU^{\dagger}=\left( \begin{array}{ccc}
 0.4  &  0\\
  0  &  0.25
 \end{array} \right) \geq 0$. But $U^{\dagger}V=\left( \begin{array}{ccc}
 0.2167   & 0.0333   & 0.1833\\
    0.0833  &  0.1667 &  -0.0833\\
    0.1333 &  -0.1333  &  0.2667
 \end{array} \right) \ngeq 0$ . Hence $A=U-V$ is a proper weak regular splitting of type II
  but not type I with
  $\rho(U^{\dagger}V)=0.4$.
\end{ex}

Another remark drawn from the above example is that it cannot be
ensured convergence of all splittings by the known convergence
results for the proper weak regular splitting of type I. To overcome
this issue, Mishra and Sivakumar \cite{misoam} proved the following
convergence result for the proper weak
  regular splitting of type II. Note that the same authors call it as the weak pseudo regular
  splitting, but we call it here as the proper weak regular splitting of type II.

\begin{thm}\label{m2} \textnormal{(Remark 3.5, \cite{misoam})}\\
Let $A=U-V$ be a proper weak regular splitting of type II of $A\in
{\R}^{m \times n}$. Then $A^{\dagger} \geq 0$ if and only if
$\rho(U^{\dagger}V)<1$.
\end{thm}

Observe that Theorem \ref{m1} and Theorem \ref{m2} together extend
Theorem 3.4 (i), \cite{e1} for rectangular matrices.   The first
main result, presented below partially generalizes the other part of
Theorem 3.4, \cite{e1}.

\begin{lem}\label{main1}
 Let  $A=U-V$ be a proper weak regular splitting of type II of a
semimonotone matrix $A \in {\R}^{m \times n}$. Suppose  that
$\rho(U^{\dagger}V)> 0$. Then there exists a  vector $x \gneq 0$
such that $U^{\dagger}Vx=\rho(U^{\dagger}V)x$, $Ax \gneq 0$ and
$Vx\gneq 0$.
\end{lem}

\begin{proof}
We have $VU^{\dagger} \geq 0$. By Theorem \ref{perron}, there exists
an  eigenvector $z \geq 0$ such that
\begin{equation}\label{eq1}
VU^{\dagger}z=\rho(VU^{\dagger})z.
\end{equation}
Therefore, $z \in  R(V) \subseteq R(U)$. Define $x=U^{\dagger}z$.
Then $x \geq 0$. Pre-multiplying (\ref{eq1}) by $U^{\dagger}$, we
obtain
\begin{equation}\label{eq2}
U^{\dagger}Vx=\rho(VU^{\dagger})x.
\end{equation}
Suppose that $x =0$. Then $U^{\dagger}z=0$ so that $z\in R(U)\cap
N(U^T)$. Thus, $z=0$, a contradiction. So $x \neq 0$. Now we prove
the inequality $Ax \geq 0$. Theorem \ref{cli} and Theorem \ref{m2}
yield
$$0 \leq
(1-\rho(VU^{\dagger}))z=(I-VU^{\dagger})z=(I-VU^{\dagger})Ux= Ax.
$$ Clearly, $Ax \neq 0$ otherwise $Ax=0$ implies $x = 0$, a contradiction. From (\ref{eq1}), we have $Vx\geq 0$.
Pre-multiplying (\ref{eq2}) by $U$, we get
$Vx=\rho(U^{\dagger}V)Ux$, i.e.,  $\displaystyle
Ux=\frac{Vx}{\rho(U^{\dagger}V)} $. Therefore, we get
$$0\leq Ax=U(I-U^{\dagger}V)x=(1-\rho(U^{\dagger}V))Ux=\frac{(1-\rho(U^{\dagger}V))}{\rho(U^{\dagger}V)}Vx .$$
So $Vx \neq 0$. If $Vx=0$, then $Ax=0$, again a contradiction.
\end{proof}

%
%

Convergence of an iteration scheme is usually accelerated by a
preconditioner.  It is a square matrix $Q$ of order $m$ which on
pre-multiplication makes the convergence of the iterative method for
the system with the matrix $QA$
  faster than the original system with the matrix $A$.
Hence, instead of solving (\ref{introeq1}), we solve
 $$QAx = Qb, i.e., A_1x=c.$$
 The method of finding of an effective preconditioner $Q$ for general problems is a mathematical challenge.
  Nevertheless, many specific problems are being successfully solved
  using
preconditioned iterative solvers.  But the problem is how to choose
an effective preconditioner. This is settled next, with a comparison
result of the rate of convergence of two different linear systems.
  The proof adopts  similar techniques as used in Theorem 3.5,
  \cite{e1}.

\begin{thm}\label{main2}
 Let $A_{1}, A_{2} \in {\R}^{m \times n}$. Let $A_{1}=U_{1}-V$ and $A_{2}=U_{2}-V$ be two
 proper weak regular splittings
  of different types. Suppose  that
$\rho(U_1^{\dagger}V)> 0$ and $\rho(U_2^{\dagger}V)> 0$. If $V \neq
0$ and $A_{2}^{\dagger}
> A_{1}^{\dagger}\geq 0$, then
   $$\rho(U_{1}^{\dagger}V)<\rho(U_{2}^{\dagger}V)<1.$$
\end{thm}

\begin{proof}
By Theorem \ref{m1} and Theorem \ref{m2}, it follows that
$\rho(U_{i}^{\dagger}V)<1$ for each $i=1,2$. Define
$G_{1}=A_{1}^{\dagger}V,~G_{2}=A_{2}^{\dagger}V,~
\tilde{G_{1}}=VA_{1}^{\dagger}$ and $\tilde{G_{2}} =VA_{2}^{\dagger}$.
Using Theorem \ref{bpthm} (c) and Theorem \ref{cli} (b), we have
$$~~~~~~G_{i}=A_{i}^{\dagger}V=(I-U_{i}^{\dagger}V)^{-1}U_{i}^{\dagger}V,~i=1,2$$
$$\text{and}~~\tilde{G_{i}}=VA_{i}^{\dagger}=VU_{i}^{\dagger}(I-VU_{i}^{\dagger})^{-1},~i=1,2.$$
Let us first assume that $A_{1}=U_{1}-V$ is a proper weak regular splitting of type I and $A_{2}=U_{2}-V$
 is a proper weak regular splitting of type II. Then
$G_{1}$ and
 $\tilde{G_{2}}$ are non-negative matrices and
$$\displaystyle \rho(G_{i})=\rho(\tilde{G_{i}})=\frac{\rho(U_{i}^{\dagger}V)}{1-\rho(U_{i}^{\dagger}V)}
=\frac{\rho(VU_{i}^{\dagger})}{1-\rho(VU_{i}^{\dagger})} \quad
\text{for each}~i=1,2.$$ We only need to show that
$\rho(G_{2})<\rho(G_{1})$. By Lemma \ref{main1}, there exists an
eigenvector $x\geq 0,$ such that
 $U_{2}^{\dagger}Vx=\rho(U_{2}^{\dagger}V)x$ and $Vx \geq 0$. Using $A_{2}^{\dagger}>A_{1}^{\dagger}\geq 0$,
  we get
\begin{equation}\label{maineq1}
\rho(G_2)x=G_{2}x=A_{2}^{\dagger}Vx>A_{1}^{\dagger}Vx=G_{1}x.
\end{equation}
Hence, by Lemma \ref{lemma3.2} $(ii)$, the strict inequality
$\rho(G_{1})<\rho(G_{2})$ follows directly. If $A_{1}=U_{1}-V$ is a
proper weak regular splitting of type II and $A_{2}=U_{2}-V$ is a proper weak
regular splitting of type I, then $\tilde{G_{1}}$ and $G_{2}$ are
 non-negative matrices. Again, by Lemma \ref{main1}, there exists an eigenvector $z\geq 0$ such that
  $U_{1}^{\dagger}Vz=\rho(U_{1}^{\dagger}V)z$ and $Vz \geq 0$. Thus
\begin{equation}\label{maineq2}
G_{2}z=A_{2}^{\dagger}Vz>A_{1}^{\dagger}Vz=G_{1}z=\rho(G_{1})z.
\end{equation}
 The strict inequality $\rho(G_{1})<\rho(G_{2})$ then follows from Lemma \ref{lemma3.2} $(i)$
  which yields the desired claim.
\end{proof}

 In the above result, one cannot drop the assumption $A_{2}^{\dagger} > A_{1}^{\dagger}\geq 0$ which can be seen
  from the example illustrated next.

\begin{ex}
Let $A_{1}=\left(\begin{array}{ccc}
7 & -7/2 & 7\\
0 & 1 & 0
\end{array}
\right)=\left(\begin{array}{ccc}
8 & -4 & 8\\
0 & 2 & 0
\end{array}
\right)-\left(\begin{array}{ccc}
1 & -1/2 & 1\\
0 & 1 & 0
\end{array}
\right)=U_{1}-V$ and $A_{2}=\left(\begin{array}{ccc}
3 & -3/2 & 3\\
0 & 1 & 0
\end{array}
\right)=\left(\begin{array}{ccc}
4 & -2 & 4\\
0 & 2 & 0
\end{array}
\right)-\left(\begin{array}{ccc}
1 & -1/2 & 1\\
0 & 1 & 0
\end{array}
\right)=U_{2}-V$. Then $U_{1}^{\dagger}=\left(\begin{array}{ccc}
 0.0625  &  0.1250\\
         0 &   0.5000\\
    0.0625  &  0.1250
\end{array}
\right) \geq 0,~U_{1}^{\dagger}V=\left(\begin{array}{ccc}
 0.0625   & 0.0937  &  0.0625\\
 0 &   0.5000 &        0\\
0.0625  &  0.0937  &  0.0625
\end{array}
\right) \geq 0,~ U_{2}^{\dagger}=\left(\begin{array}{cc}
0.1250 &   0.1250\\
   0 &   0.5000\\
    0.1250  &  0.1250
\end{array}
\right) \geq 0$ and $VU_{2}^{\dagger}=\left(\begin{array}{ccc}
       0.2500 &   0\\
   0  &  0.5000
\end{array}
\right)\geq 0.$ So $A_{1}=U_{1}-V$ is a proper weak regular
splitting of type I and $A_{2}=U_{2}-V$ is a proper weak regular
splitting of type II. We have
 $A_{2}^{\dagger}=\left(\begin{array}{ccc}
 0.1667 &   0.2500\\
    0  &  1\\
    0.1667  &  0.2500
\end{array}
\right)  \geq
A_{1}^{\dagger}=\left(\begin{array}{ccc}
 0.0714  &  0.2500\\
  0 &   1\\
    0.0714 &   0.2500
\end{array}
\right) \geq 0$. But
$\rho(U_{1}^{\dagger}V)=\rho(U_{2}^{\dagger}V)=0.5.$
\end{ex}

For the square nonsingular case,  we have the following corollary to
Theorem \ref{main2}.

\begin{cor}\textnormal{(\cite{e1})}\\
 Let $A_{1}, A_{2} \in {\R}^{n \times n}$.  Let $A_{1}=U_{1}-V$ and $A_{2}=U_{2}-V$ be two
 weak regular splittings
 of different types. Suppose  that
$\rho(U_1^{-1}V)> 0$ and $\rho(U_2^{-1}V)> 0$.
   If $V \neq 0$ and $A_{2}^{-1}>A_{1}^{-1} \geq 0$,
    then $$\rho(U_{1}^{-1}V)<\rho(U_{2}^{-1}V)<1.$$
\end{cor}

We remark that the above result partially extends Theorem 3.5, \cite{e1}. While the authors assumed different
 hypotheses $A_{1} \lneq A_{2}$,  $A_{1}^{-1} > 0$  and $A_{2}^{-1}\geq 0$ in Theorem 3.5, \cite{e1}, we assumed
$A_{2}^{-1}>A_{1}^{-1} \geq 0$ in place of these three conditions. This is also mentioned after the proof
of Theorem 3.5 of \cite{e1}.
We conclude this section with another comparison theorem for two
different linear systems having two different types
 of proper weak regular splittings.

\begin{thm}\label{main5}
Let $A_{1}, A_{2} \in {\R}^{m \times n}$. Let $A_{1}=U_{1}-V_{1}$
and $A_{2}=U_{2}-V_{2}$ be two proper weak regular splittings of
different types. Suppose  that $\rho(U_1^{\dagger}V_{1})> 0$ and
$\rho(U_2^{\dagger}V_{2})> 0$. Assume that
    $V_{1} \neq 0$, $V_{2} \neq 0$ and $A_{2}^{\dagger}>A_{1}^{\dagger} \geq 0$. If
$V_{1} \leq V_{2},$ then
$$\rho(U_{1}^{\dagger}V_{1}) < \rho(U_{2}^{\dagger}V_{2})<1.$$

\end{thm}
\begin{proof}
By Theorem \ref{m1} and Theorem  \ref{m2},
 we obtain $\rho(U_{i}^{\dagger}V_{i})<1,~i=1,2$.
The remaining proof is similar to the proof of Theorem \ref{main2},
with the exception that in place of (\ref{maineq1})
 we have to use one additional inequality
 $$\rho(G_2)x=G_{2}x=A_{2}^{\dagger}V_{2}x > A_{1}^{\dagger}V_{1}x=G_{1}x,$$
  and in place of (\ref{maineq2}), we need $G_{2}z = A_{2}^{\dagger}V_{2}z > A_{1}^{\dagger}V_{1}z=G_{1}z=\rho(G_{1})z$.
\end{proof}

Note that Theorem \ref{main2} is a special case of the
above result as the assumption $V_{1} \leq V_{2}$ is automatically fulfilled when $V_{1}=V_{2}$.\\

The example given below demonstrates that the converse of the above
theorem is not true.

\begin{ex}
Let $A_{1}=\left( \begin{array}{ccc}
2& -2 & 4\\
2 & 4 & -2\\
\end{array}
\right)$ and $A_{2}=\left( \begin{array}{ccc}
1& -2& 3\\
1& 3 & -2\\
\end{array}
\right)$. Then  $A_{2}^{\dagger}=\left( \begin{array}{cc}
 0.3333 &   0.3333\\
    0.0667 &   0.2667\\
    0.2667  &  0.0667
\end{array}
\right) > A_{1}^{\dagger}=\left( \begin{array}{cc}
 0.1667  &  0.1667\\
   0 &   0.1667\\
    0.1667  &  0
\end{array}
\right)\geq 0$. Let $U_{1}=\left( \begin{array}{ccc}
3 & -3 & 6\\
2 & 4 & -2
\end{array}
\right)$  and $U_{2}=\left( \begin{array}{ccc}
2 & -2 & 4\\
2 & 4 & -2
\end{array}
\right).$ Then $A_{1}=U_{1}-V_{1}$ is a proper weak regular
splitting of type I and $A_{2}=U_{2}-V_{2}$ is a proper weak regular
splittings of type II. We  have $0.3=\rho(U_{1}^{\dagger}V_{1}) <
0.5=\rho(U_{2}^{\dagger}V_{2})<1.$ But $V_{1}=\left(
\begin{array}{ccc}
1 & -1 & 2\\
0 & 0 & 0
\end{array}
\right) \nleq V_{2}=\left( \begin{array}{ccc}
1 & 0 & 1\\
1 & 1 & 0
\end{array}
\right)$.
\end{ex}

\section{Proper nonnegative splittings of different types}
The plan of this section is to obtain new comparison results for
proper nonnegative splittings of different types in order to speed
up the rate of  convergence of the iteration scheme
(\ref{introeq2}). The class of proper nonnegative splittings
contains earlier two classes of splittings, and hence study of this
class of matrices assumes significance. For later use, we record
first the following convergence result.

\begin{lem}\textnormal{(Lemma 3.5, \cite{miscma})}\label{lemd1}\\
Let $A=U-V$ be a proper nonnegative splitting of $A\in {\R}^{m
\times n}$. Then $A^{\dagger}V \geq 0$ if and only if $\displaystyle
\rho(U^{\dagger}V)=
\frac{\rho(A^{\dagger}V)}{1+\rho(A^{\dagger}V)}<1.$
\end{lem}

 Next, we recollect the definition of a proper nonnegative splitting of type II
 proposed by Baliarsingh and Mishra \cite{alekha2}. Note that the proper nonnegative
  splitting of type I is same as the proper nonnegative splitting.

\begin{dfn} \textnormal{(Definition 3.14, \cite{alekha2})}\\
 A proper splitting $A=U-V$  of $A \in {\R}^{m \times n}$ is called a proper nonnegative
  splitting of type II if $VU^{\dagger}\geq 0$.
\end{dfn}

A convergence result for a proper nonnegative splitting of type II
is stated next.

\begin{lem}\textnormal{(Remark 2, \cite{c1})}\\
Let $A=U-V$ be a proper nonnegative splitting of type II of $A\in {\R}^{m
\times n}$. Then $VA^{\dagger} \geq 0$ if and only if $\displaystyle
\rho(VU^{\dagger})=
\frac{\rho(VA^{\dagger})}{1+\rho(VA^{\dagger})}<1.$
\end{lem}

 We now prove the following comparison result which extends
 a part of Theorem 2.11, \cite{song} to rectangular matrices.

\begin{thm}\label{main6}
Let $A=U_{1}-V_{1}=U_{2}-V_{2}$ be two convergent proper nonnegative
splittings of the same
 type of a semimonotone matrix $A \in {\R}^{m \times n}$. If there exists $ \alpha,~ 0<\alpha \leq 1,$
 such that $V_{1}\leq \alpha V_{2}$ and $\rho(A^{\dagger}V_{i})>0,~i=1 ~ or~2,$
  then $$\rho(U_{1}^{\dagger}V_{1}) \leq \rho(U_{2}^{\dagger}V_{2})<1,$$
  whenever $\alpha =1$ and $$\rho(U_{1}^{\dagger}V_{1}) < \rho(U_{2}^{\dagger}V_{2})<1,$$
  whenever $0< \alpha < 1.$
\end{thm}

\begin{proof}
Assume that the given splittings are convergent proper nonnegative
splittings of type I. So, we have $\rho(U_1^{\dagger}V_{1})<1$. By
Lemma \ref{lemd1}, we get $A^{\dagger}V_{1}\geq 0$. The
 conditions $A^{\dagger} \geq 0$ and $V_{1}\leq \alpha V_{2}$ together
 imply
$$0 \leq A^{\dagger}V_{1} \leq \alpha A^{\dagger}V_{2}.$$
It then follows from Theorem \ref{var1} that
\begin{equation}\label{ee1}
\rho( A^{\dagger}V_{1}) \leq \alpha \rho(A^{\dagger}V_{2}).
\end{equation}
Since $\displaystyle f(\eta)=\frac{\eta}{1+\eta}$ is a
strictly increasing
 function for $\eta \geq 0$, so
$$\frac{\rho(A^{\dagger}V_{1})}{1+\rho(A^{\dagger}V_{1})} \leq \frac{\alpha \rho(A^{\dagger}V_{2})}{1+\alpha\rho(A^{\dagger}V_{2})}.$$
For $\alpha=1$, the required result follows from Lemma \ref{lemqeigval},
 since $\displaystyle \rho(U^{\dagger}_{i}V_{i})=\frac{\rho(A^{\dagger}V_{i})}{1+\rho(A^{\dagger}V_{i})}>0$
 for $i=1$ or $ 2$. If $0< \alpha <1$, then from (\ref{ee1}), we get
 $$\rho( A^{\dagger}V_{1})<\rho(A^{\dagger}V_{2}),$$
and proceeding as before, we get the desire result.

The proof goes parallel in the case of proper nonnegative splitting of type
II.
\end{proof}

The second part of Theorem 2.11, \cite{song} is obtained as a corollary to the above result.
\begin{cor}
Let $A=U_{1}-V_{1}=U_{2}-V_{2}$ be two convergent  nonnegative
splittings of the same
 type of a monotone matrix $A \in {\R}^{n \times n}$. If there exists $ \alpha,~ 0<\alpha \leq 1,$
 such that $V_{1}\leq \alpha V_{2}$ and $\rho(A^{-1}V_{i})>0,~i=1 ~ or~2,$
  then $$\rho(U_{1}^{-1}V_{1}) \leq \rho(U_{2}^{-1}V_{2})<1,$$
  whenever $\alpha =1$ and $$\rho(U_{1}^{-1}V_{1}) < \rho(U_{2}^{-1}V_{2})<1,$$
  whenever $0< \alpha < 1.$
\end{cor}

 In the case of proper nonnegative splittings of different types, the
 following result can be proved in a similar way as of the above
 one which extends Theorem 2.12, \cite{song} for rectangular matrices.
\begin{thm}\label{main7}
Let $A=U_{1}-V_{1}=U_{2}-V_{2}$ be two convergent proper nonnegative
splittings of different
 types of a semimonotone matrix $A \in {\R}^{m \times n}$. If there exists $ \alpha,~ 0<\alpha \leq 1,$
 such that $V_{1}\leq \alpha V_{2}$ and $\rho(A^{\dagger}V_{i})>0,~i=1 ~ or~2,$
  then $$\rho(U_{1}^{\dagger}V_{1}) \leq \rho(U_{2}^{\dagger}V_{2})<1,$$
  whenever $\alpha =1$ and $$\rho(U_{1}^{\dagger}V_{1}) < \rho(U_{2}^{\dagger}V_{2})<1,$$
  whenever $0< \alpha < 1.$
\end{thm}

Another comparison result for proper nonnegative splittings of
different types is established below.

\begin{thm}\label{main8}
Let $A=U_{1}-V_{1}=U_{2}-V_{2}$ be two convergent proper nonnegative
splittings of different types of a semimonotone matrix $A \in
{\R}^{m \times n}$. If there exists $0< \alpha \leq 1,$ such that
$U_{2}^{\dagger}\leq \alpha U_{1}^{\dagger},$ then
$$\rho(U_{1}^{\dagger}V_{1}) \leq \rho(U_{2}^{\dagger}V_{2})<1,$$
whenever $\alpha =1$ and  $$\rho(U_{1}^{\dagger}V_{1}) < \rho(U_{2}^{\dagger}V_{2})<1,$$
whenever $0<\alpha <1.$
\end{thm}

\begin{proof}
Assume that $A=U_{1}-V_{1}$ is a convergent proper nonnegative
splitting of type I and $A=U_{2}-V_{2}$
 is  a convergent proper nonnegative splitting of type II.
It then follows from Theorem \ref{var2}
that $(I-U_{1}^{\dagger}V_{1})^{-1}\geq 0$ and
  $(I-V_{2}U_{2}^{\dagger})^{-1}\geq 0$, respectively. By using Theorem \ref{cli} and
  the given condition $U_{2}^{\dagger} \leq \alpha U_{1}^{\dagger}$, we have
\begin{equation}\label{eq9}
A^{\dagger}=U_{2}^{\dagger}(I-V_{2}U_{2}^{\dagger})^{-1}\leq \alpha
U_{1}^{\dagger}(I-V_{2}U_{2}^{\dagger})^{-1}.
\end{equation}
Pre-multiplying (\ref{eq9}) by $(I-U_{1}^{\dagger}V_{1})^{-1}$, we get
\begin{equation}\label{eq10}
(I-U_{1}^{\dagger}V_{1})^{-1}A^{\dagger}\leq \alpha
(I-U_{1}^{\dagger}V_{1})^{-1}U_{1}^{\dagger}
(I-V_{2}U_{2}^{\dagger})=\alpha
A^{\dagger}(I-V_{2}U_{2}^{\dagger})^{-1}.
\end{equation}
Since $U^{\dagger}_{1}V_{1}\geq 0$,  there exists an eigenvector
 $x \geq 0$ such that
\begin{equation}\label{eq101}
 x^{T}U_{1}^{\dagger}V_{1}=\rho(U^{\dagger}_{1}V_{1})x^{T},
\end{equation}
by Theorem \ref{perron}. So $x\in R(V_{1}^T)\subseteq   R(A^T)$.
Pre-multiplying (\ref{eq10}) by $x^{T}$, we get
$$\frac{1}{1-\rho(U_{1}^{\dagger}V_{1})}x^{T}A^{\dagger}
\leq  \alpha x^{T}A^{\dagger}(I-V_{2}U^{\dagger}_{2})^{-1}.$$ By
Lemma \ref{lemma3.2}, it then follows that
$$\frac{1}{1-\rho(U_{1}^{\dagger}V_{1})} \leq \frac{\alpha
}{1-\rho(V_{2}U_{2}^{\dagger})}=\frac{\alpha
}{1-\rho(U_{2}^{\dagger}V_{2})},$$
i.e,
\begin{equation}\label{eq102}
\rho(U_{2}^{\dagger}V_{2}) \geq (1-\alpha) + \alpha \rho(U_{1}^{\dagger}V_{1}).
\end{equation}
As $x^{T}A^{\dagger} \geq 0$
and $x^{T}A^{\dagger} \neq 0$. Suppose that $x^{T}A^{\dagger} = 0$,
then $x^{T}A^{\dagger}A = 0$, i.e.,
$(A^{\dagger}A)^{T}x=A^{\dagger}Ax=x = 0$, a contradiction. Hence $x^{T}A^{\dagger} \neq 0$.
Now, the desired result follows immediately  from (\ref{eq102}).\\
In the case of  $A=U_{1}-V_{1}$ is a proper nonnegative
 splitting of type II and $A=U_{2}-V_{2}$ is a proper nonnegative splitting of type
 I, the proof is analogous to the above proof.
\end{proof}

The following is an immediate consequence of the above  result when
square nonsingular matrices are considered, and is a part of Theorem 2.14, \cite{song}.

\begin{cor}
Let $A=U_{1}-V_{1}=U_{2}-V_{2}$ be two convergent  nonnegative
splittings of different types of a monotone matrix $A \in
{\R}^{n \times n}$. If there exists $0< \alpha \leq 1,$ such that
$U_{2}^{-1}\leq \alpha U_{1}^{-1},$ then
$$\rho(U_{1}^{-1}V_{1}) \leq \rho(U_{2}^{-1}V_{2})<1,$$
whenever $\alpha =1$ and  $$\rho(U_{1}^{-1}V_{1}) < \rho(U_{2}^{-1}V_{2})<1,$$
whenever $0<\alpha <1.$
\end{cor}

The next result addresses the question of  existence of  an
$\alpha$.

\begin{thm}\label{main9}
 Let $A=U_{1}-V_{1}=U_{2}-V_{2}$ be
two convergent proper nonnegative splittings of different types of a
semimonotone matrix $A \in {\R}^{m \times n}$. If $U_{1}^{\dagger} > U_{2}^{\dagger},$
then there
 exists $\alpha,~0<\alpha<1,$ such that $U_{2}^{\dagger}\leq \alpha U_{1}^{\dagger} $
 and $\rho(U_{1}^{\dagger}V_{1})<\rho(U_{2}^{\dagger}V_{2})<1$.
\end{thm}

\begin{proof}
Denote $$U_{1}^{\dagger}=(a_{ij}), \quad U_{2}^{\dagger}=(b_{ij}),
~~i=1,2,\ldots, n,~~j=1,2,\ldots,m.$$ From $U_{1}^{\dagger} >
U_{2}^{\dagger}$, we get
$$a_{ij}>b_{ij},\quad i=1,2,\ldots,n,~~j=1,2,\ldots, m.$$
If there exists $b_{ij}>0$ for some $i,j$, then let
$\alpha=\displaystyle \max_{\substack{
0 \leq i \leq n \\
0 \leq j \leq n }}\left \{\frac{b_{ij}}{a_{ij}}\vert~ b_{ij}>0
\right\}, $ otherwise, $0<\alpha < 1$ is arbitrary. Clearly, $0 <
\alpha < 1$ and
$$b_{ij}\leq \alpha ~a_{ij},\quad i=1,2,\ldots, n,~~j=1,2,\ldots,m,$$
i.e., $$U_{2}^{\dagger} \leq \alpha U_{1}^{\dagger}.$$ By Theorem
\ref{main8}, the inequality follows.
\end{proof}

Corollary 2.15, \cite{song} is obtained next as a corollary to the
above result in the case of square nonsingular matrices.

\begin{cor}\textnormal{(Corollary 2.15, \cite{song})}\\
Let $A=U_{1}-V_{1}=U_{2}-V_{2}$ be two convergent  nonnegative
splittings of different types of a monotone matrix $A \in
{\R}^{n \times n}$. If $U_{1}^{-1} > U_{2}^{-1},$ then
there
 exists $\alpha,~ 0<\alpha<1,$ such that $U_{2}^{-1}\leq \alpha U_{1}^{-1} $
 and $\rho(U_{1}^{-1}V_{1})<\rho(U_{2}^{-1}V_{2})<1$.
\end{cor}

The example given below demonstrates that the converse of Theorem
\ref{main9} is not true.

\begin{ex}
Let $A=\left(\begin{array}{ccc}
5 & -4 & 0\\
-7 & 7 & 0
\end{array}\right)$. Then $A^{\dagger}=\left(\begin{array}{cc}
1 &   0.5714\\
    1 &   0.7143\\
         0  &       0
\end{array}
\right)\geq 0$. Let $U_{1}=\left(\begin{array}{ccc}
5 & -1 & 0\\
-7 & 7 & 0
\end{array}
\right)$ and $U_{2}=\left(\begin{array}{ccc}
5 & 0 & 0\\
0 & 8 & 0
\end{array}
\right).$ Then  $A=U_{1}-V_{1}$ is a proper nonnegative splitting of
type I and $A=U_{2}-V_{2}$ is a proper nonnegative splitting of type
II. We have
  $0.7500=\rho(U^{\dagger}_{1}V_{1})<\rho(U_{2}^{\dagger}V_{2})=0.9015<1$,
   and for $\displaystyle\alpha=0.8$,
 $U_{2}^{\dagger}=\left(\begin{array}{cc}
 0.2000   &      0\\
         0 &   0.1250\\
         0 &        0
\end{array} \right) \leq \left( \begin{array}{cc}
0.2000 &   0.0286\\
    0.2000 &   0.1429\\
         0   &      0
\end{array}\right)=\alpha U_{1}^{\dagger}$.
  But  $U_{1}^{\dagger}=\left( \begin{array}{cc}
 0.2500  &  0.0357\\
    0.2500 &   0.1786\\
         0  &       0
\end{array}\right)\ngtr\left(\begin{array}{cc}
0.2000    &     0\\
         0 &   0.1250\\
         0   &      0
\end{array} \right)= U_{2}^{\dagger}$.
\end{ex}

The following example shows that Theorem \ref{main8} and Theorem
\ref{main9} do not valid, if we consider proper nonnegative
splittings of same types instead of different types.

\begin{ex}
Let $A=\left(\begin{array}{ccc}
3 & -2 & 3\\
-2 & 3 & -2
\end{array}\right)$. Then $A^{\dagger}=\left(\begin{array}{cc}
3/10 & 1/5 \\
2/5 & 3/5\\
3/10 & 1/5
\end{array} \right)> 0$. Let $U_{1}=\left(\begin{array}{ccc}
12 & -10 & 12\\
-8 & 15 & -8
\end{array}
\right)$ and $U_{2}=\left(\begin{array}{ccc}
25/2 & -10 & 25/2\\
-8 & 15 & -8
\end{array}\right).$ Then
$A=U_{1}-V_{1}=U_{2}-V_{2}$ are two convergent proper nonnegative
splittings of type I.  We  have
$U^{\dagger}_{1}=\left(\begin{array}{cc}
0.0750  &  0.0500\\
    0.0800 &   0.1200\\
    0.0750  &  0.0500
\end{array} \right)>U^{\dagger}_{2}=\left(\begin{array}{cc}
 0.0698 &   0.0465\\
    0.0744 &   0.1163\\
    0.0698  &  0.0465
\end{array} \right),$
and for $\alpha= 0.9690<1$,
$U^{\dagger}_{2}=
\left(\begin{array}{cc}
 0.0698   & 0.0465\\
    0.0744 &   0.1163\\
    0.0698  &  0.0465
\end{array} \right) \leq \left(\begin{array}{cc}
 0.0727  &  0.0484\\
    0.0775 &   0.1163\\
    0.0727 &   0.0484
\end{array} \right)=\alpha U^{\dagger}_{1}.$
But $\rho(U_{1}^{\dagger}V_{1})=\rho(U_{2}^{\dagger}V_{2})=0.8.$
\end{ex}

The condition $A^{\dagger} \geq 0$ in Theorem \ref{main8} and
Theorem \ref{main9} is not redundant, and is illustrated hereunder
by an example.

\begin{ex}
Let $A=\left(\begin{array}{ccc}
2 & -7 & 2\\
-8 & 5 & -8
\end{array}
\right).$ Then $A^{\dagger}=\left(\begin{array}{cc}
 -0.0543  & -0.0761\\
   -0.1739 &  -0.0435\\
   -0.0543  & -0.0761

\end{array} \right) < 0.$ Let $U_{1}=\left(\begin{array}{ccc}
4 & -35 & 4\\
-16 & 25 & -16
\end{array}
\right)$ and
$U_{2}=\left(\begin{array}{ccc}
3 & -21/2 & 3\\
-12 & 15/2 & -12
\end{array}
\right)$. Then  $A=U_{1}-V_{1}$ is a proper nonnegative splitting of
type I
 and $A=U_{2}-V_{2}$ is a proper nonnegative splitting of type II.
  We have $0.3333=\rho(U_{2}^{\dagger}V_{2}) <
\rho(U_{1}^{\dagger}V_{1})=0.8.$
 But
 $U^{\dagger}_{2}=\left(\begin{array}{cc}
 -0.0362 &  -0.0507\\
   -0.1159 &  -0.0290\\
   -0.0362 &  -0.0507

\end{array} \right) < \left(\begin{array}{cc}
 -0.0272  & -0.0380\\
   -0.0348  & -0.0087\\
   -0.0272  & -0.0380\\
\end{array} \right)= U^{\dagger}_{1}.$
\end{ex}

The above example also motivates us to prove the following theorem
which extends Theorem 2.4, \cite{ziw2} to rectangular matrices.
However, we provide below a  short new proof.

\begin{thm}
Let $A=U_{1}-V_{1}=U_{2}-V_{2}$ be two convergent proper nonnegative
splittings of different types of $A\in{\R}^{m \times n}$. If
$A^{\dagger} \leq 0$ and $U_{2}^{\dagger} \geq U_{1}^{\dagger}$,
 then $$\rho(U_{1}^{\dagger}V_{1}) \leq \rho(U_{2}^{\dagger}V_{2})<1.$$
In particular, if
$A^{\dagger} < 0$ and $U_{2}^{\dagger} > U_{1}^{\dagger}$,
 then $$\rho(U_{1}^{\dagger}V_{1})<\rho(U_{2}^{\dagger}V_{2})<1.$$
\end{thm}

\begin{proof}
Assume that $A=U_{1}-V_{1}$ is a proper nonnegative of type I and
$A=U_{2}-V_{2}$ is a proper nonnegative of type II. Then there exists
an eigenvector $x \geq 0$ such that
\begin{equation}\label{n1}
x^{T}U_{1}^{\dagger}V_{1}=\rho(U_{1}^{\dagger}V_{1})x^{T}
\end{equation}
Therefore, $x \in  R(V_{1}^{T}) \subseteq R(U_{1}^{T})=R(A^{T})$.  From the given condition
 $U_{2}^{\dagger} \geq U_{1}^{\dagger}$, we obtain the  following inequality
\begin{equation}\label{eqq1}
A^{\dagger}=U_{2}^{\dagger}(I-V_{2}U_{2}^{\dag})^{-1} \geq U_{1}^{\dagger}(I-V_{2}U_{2}^{\dag})^{-1}.
\end{equation}
Pre-multiplying (\ref{eqq1}) by $(I-U_{1}^{\dagger}V_{1})^{-1}$, we obtain
\begin{equation}\label{eqq2}
(I-U_{1}^{\dagger}V_{1})^{-1}A^{\dagger} \geq
 (I-U_{1}^{\dagger}V_{1})^{-1}U_{1}^{\dagger}(I-V_{2}U_{2}^{\dag})^{-1}
 =A^{\dagger}(I-V_{2}U_{2}^{\dag})^{-1}.
\end{equation}
Again, pre-multiplying (\ref{eqq2}) by $x^{T}$, we get
\begin{equation}\label{eqq3}
\displaystyle \frac{1}{1-\rho(U_{1}^{\dagger}V_{1})}x^{T}A^{\dagger}
\geq x^{T}A^{\dagger}(I-V_{2}U_{2}^{\dag})^{-1}.
\end{equation}
Let $z=x^{T}A^{\dagger}$. Clearly, $z \leq 0$ and $z \neq 0$.
Otherwise, $x \in R(A^{T}) \cap N(A)$, which is a contradiction. So,
we get
$$\displaystyle \frac{1}{1-\rho(U_{1}^{\dagger}V_{1})}(-z) \leq
(-z)(I-V_{2}U_{2}^{\dag})^{-1}.$$ Now, the required result follows
from Lemma \ref{lemma3.2}.

The proof follows similarly when $A=U_{1}-V_{1}$ is proper
nonnegative of type II and $A=U_{1}-V_{1}$ is proper nonnegative of
type I.

\end{proof}

 Theorem 2.4, \cite{ziw2} is obtained as a corollary to the above result.

\begin{cor}\textnormal{(Theorem 2.4, \cite{ziw2})}\\
Let $A=U_{1}-V_{1}=U_{2}-V_{2}$ be two convergent nonnegative
splittings of different types of
$A\in{\R}^{n \times n}$. If $A^{-1} \leq 0$ and $U_{2}^{-1} \geq U_{1}^{-1}$, then
$$\rho(U_{1}^{-1}V_{1})<\rho(U_{2}^{-1}V_{2})<1.$$
In particular, if $A^{-1} < 0$ and $U_{2}^{-1} > U_{1}^{-1}$, then
 $$\rho(U_{1}^{-1}V_{1})<\rho(U_{2}^{-1}V_{2})<1.$$
\end{cor}

\section{Comparison of proper multisplittings}
Improving the rate of convergence of the iteration scheme
(\ref{introeq2}) is a problem of interest for getting the solution
faster. In this direction, Climent and Perea \cite{c3} have proposed
proper multisplitting theory for rectangular matrices  while the
authors of \cite{white}  have studied the same problem in the
nonsingular matrix setting.  Here, we revisit the same theory as
proposed by Climent and Perea \cite{c3} first, and then produced a
few new convergence and comparison theorems for proper
multisplittings. In this context, the definition of a proper
multisplitting is recalled below.

\begin{dfn}\textnormal{(Definition 2, \cite{c3})}\\ \label{multisplitting1}
The triplet ${(U_k , V_k , E_k )}_{k=1}^{p}$
 is called a proper multisplitting of  $A \in {\R}^{m \times n}$ if\\
$(i) ~A= U_k- V_k$ is a proper splitting, for each $k=1, 2, \ldots, p$,\\
$(ii) ~E_k\geq0$, for  each $k=1, 2, \ldots, p$ is a diagonal ${m\times
m}$ matrix, and $ \sum_{k=1}^{p}E_{k} = I, $ where $ I $ is the $
m\times m$ identity matrix.
\end{dfn}

Using Definition \ref{multisplitting1}, Climent and Perea \cite{c3}
have considered
 the iteration scheme for solving (\ref{introeq1}) as follows:
\begin{equation}\label{climenteq}
x_{i+1}=Hx_{i}+Gb, \quad i=0,1,2,\ldots,
\end{equation}
where $H=\displaystyle\sum_{k=1}^{p}E_{k}U_{k}^{\dagger}V_{k}$ and
 $G=\displaystyle\sum_{k=1}^{p}E_{k}U_{k}^{\dagger}$. Here onwards, all
  $H$ and $G$ are defined as above unless stated otherwise.

\begin{remark}
Note that the matrix multiplication $E_{k}U^{\dagger}_{k}$ is not
defined in $G$ due to the order of $E_k$ is in an incorrect
form.
\end{remark}

We thus have modified the above definition,  and is presented next.

\begin{dfn}\label{multisplitting}
The triplet ${(U_k , V_k , E_k )}_{k=1}^{p}$
 is called a proper multisplitting of  $A \in {\R}^{m \times n}$ if\\
$(i) ~A= U_k- V_k$ is a proper splitting, for each $k=1, 2, \ldots, p$,\\
$(ii) ~E_k\geq0$, for each  $k=1, 2, \ldots, p$ is a diagonal
${n\times n}$ matrix, and $ \sum_{k=1}^{p}E_{k} = I, $ where $ I $
is the $n\times n$ identity matrix.
\end{dfn}

Then $H$ and $G$ are well defined.
  A proper multisplitting is called a proper regular multisplitting or a proper weak regular multisplitting,
 if each one of the proper splitting is a proper regular splitting or a proper weak regular splitting,
 respectively. Climent and Perea \cite{c3} obtained the following results for
a  proper weak regular multisplitting.

\begin{lem}\textnormal{(Lemma 1, \cite{c3})}\label{perea}\\
Let $(U_{k},V_{k},E_{k})_{k=1}^{p}$ be a proper weak regular multisplitting  of $A \in {\R}^{m \times n}$. Then\\
$(i)~H \geq 0$ and therefore $H^{j}$ for $j=0,1,\ldots.$\\
$(ii)~\displaystyle\sum_{k=1}^{p}E_{k}U_{k}^{\dagger}A=(I-H)A^{\dagger}A$.\\
$(iii)~(I+H+H^{2}+\cdots+H^{m})(I-H)=I-H^{m+1}.$
\end{lem}

\begin{thm}\textnormal{(Theorem 4, \cite{c3})}\label{pereathm}\\
Let $(U_{k},V_{k},E_{k})_{k=1}^{p}$ be a proper weak regular
multisplitting  of a semimonotone matrix $A \in {\R}^{m \times n}$.
Then $\rho(H)<1.$
\end{thm}

It is of interest to know  the type of splitting $B-C$ of $A$ that
yields the iteration scheme (\ref{climenteq}) which is  restated as
what can we say about the type of the induced splitting $A=B-C$
being induced by
$H=\displaystyle\sum_{k=1}^{p}E_{k}U_{k}^{\dagger}V_{k}$. With an
additional hypothesis $R(E_{k}) \subseteq R(A^{T})$, for each
$k=1,2,\ldots,p$,
 of a proper weak regular multisplitting, we establish the following new
 result which addresses the above issue partially.

\begin{thm}\label{multi1}
Let $(U_{k},V_{k},E_{k})_{k=1}^{p}$ be a proper weak regular
multisplitting of a semimonotone matrix $A \in {\R}^{m \times n}$. Then the unique splitting
 $A=B-C$ induced by $H$ with
$B=A(I-H)^{-1}$ is a convergent proper weak regular splitting if
$R(E_{k}) \subseteq R(A^{T}),~\text{for each}~
k=1,2,\ldots,p$.
\end{thm}

\begin{proof}
By using the condition $R(E_{k}) \subseteq R(A^{T})$, we
have $A^{\dagger}AE_{k}=E_{k}$ and $E_{k}A^{\dagger}A=E_{k}$. Then
\begin{align*}
\displaystyle A^{\dagger}AH  &=A^{\dagger}A\sum_{k=1}^{p} E_{k}U_{k}^{\dagger}V_{k}\\
&=\sum_{k=1}^{p} A^{\dagger}AE_{k}U_{k}^{\dagger}V_{k}\\
&=\sum_{k=1}^{p} E_{k}U_{k}^{\dagger}V_{k}\\
\end{align*}
\begin{align*}
&=\sum_{k=1}^{p}E_{k}U_{k}^{\dagger}V_{k}A^{\dagger}A\\
&=HA^{\dagger}A=H.
\end{align*}
Now, post-multiplying  Lemma \ref{perea}$(ii)$ by $A^{\dagger}$, we
get $G=(I-H)A^{\dagger}$. By Theorem \ref{pereathm}, we obtain
$\rho(H)<1$ and so $(I-H)$ is invertible. From equation
(\ref{climenteq}), we obtain $B^{\dagger}=G=(I-H)A^{\dagger}$.
 Let $X=A(I-H)^{-1}$.
 Then $XB^{\dagger}=AA^{\dagger}$ and $B^{\dagger}X=(I-H)A^{\dagger}A(I-H)^{-1}=
 (A^{\dagger}A-HA^{\dagger}A)(I-H)^{-1}=
 (A^{\dagger}A-A^{\dagger}AH)(I-H)^{-1}=A^{\dagger}A(I-H)(I-H)^{-1}=A^{\dagger}A$ which imply $XB^{\dagger}$
 and $B^{\dagger}X$ are symmetric. Also, $XB^{\dagger}X=AA^{\dagger}A(I-H)^{-1}=A(I-H)^{-1}=X$
 and $B^{\dagger}XB^{\dagger}=A^{\dagger}A(I-H)A^{\dagger}
 =(A^{\dagger}A-A^{\dagger}AH)A^{\dagger}=(A^{\dagger}A-HA^{\dagger}A)A^{\dagger}=(I-H)A^{\dagger}AA^{\dagger}
 =(I-H)A^{\dagger}=B^{\dagger}$. Therefore, $B=A(I-H)^{-1}$.

Clearly, $R(B)=R(A)$ as $B=A(I-H)^{-1}$. Next we prove that
$N(B)=N(A)$. Let $x \in N(A)$. Then $0=Ax=B(I-H)x
 =B(x-Hx)
 =B(x-\displaystyle\sum_{k=1}^{p}E_{k}U_{k}^{\dagger}V_{k}x)
 =Bx,$
  since $N(V_{k}) \supseteq N(A)$. So $N(A) \subseteq N(B)$. Again, let $y \in N(B)$. Then we
  get
$By=A(I-H)^{-1}y=0$. Pre-multiplying $A^{\dagger}$, we
get $A^{\dagger}A(I-H)^{-1}y=0$. Again, using the fact that
$A^{\dagger}AH=HA^{\dagger}A$ and pre-multiplying $A$, we get
 $Ay=0$. So $N(B) \subseteq
N(A)$.  Thus $N(B)=N(A).$

Next, we have to prove that $A=B-C$ is unique. Suppose that there
exists another induced splitting $A=\tilde{B}-\tilde{C}$ such that
$\tilde{B}=A(I-H)^{-1}$. Then $\tilde{B}^{\dagger}\tilde{C}=H$ and
$\tilde{B}H=\tilde{B}\tilde{B}^{\dagger}\tilde{C}=\tilde{C}=\tilde{B}-A$.
So $\tilde{B}=A+\tilde{B}H$, i.e., $\tilde{B}(I-H)=A$. This reveals
that $\tilde{B}=A(I-H)^{-1}=B$ and therefore, $H$ induces the unique
proper splitting $A=B-C$.

Finally, $B^{\dagger}=G \geq 0$ and
$B^{\dagger}C=B^{\dagger}(B-A)=B^{\dagger}B-B^{\dagger}A=
A^{\dagger}A-A^{\dagger}A(I-H)=A^{\dagger}AH=H \geq 0$. By Theorem \ref{pereathm},
 we get $\rho(B^{\dagger}C)=\rho(H)<1$.

\end{proof}

The corollary produced below adds a new convergence result to
multisplitting theory for solving the square nonsingular system of
linear equations.

\begin{cor}
Let $(U_{k},V_{k},E_{k})_{k=1}^{p}$ be a  weak regular
multisplitting of a monotone matrix $A \in {\R}^{n \times n}$. Then the unique
splitting $A=B-C$ induced by $H$ with
$B=A(I-H)^{-1}$ is a convergent  weak regular splitting.
\end{cor}

Next result says that the induced splitting is also a proper regular
splitting under the assumption of an extra condition $A \geq 0$.

\begin{thm}\label{prms}
Let $(U_{k}, V_{k},E_{k})_{k=1}^{p}$ be a proper weak regular
multisplitting of a semimonotone matrix $A \in {\R}^{m \times n}$.
Then the splitting $A=B-C$ induced by $H$ is a proper regular
splitting if $A \geq 0$ and $R(E_{k})\subseteq R(A^{T})$,
for each $k=1,2,\ldots,p$.
\end{thm}

\begin{proof}
By Theorem \ref{multi1}, the splitting $A=B-C$ induced by $H$ is
proper weak regular. Now we have to show that $C \geq 0$. So
$C=B-A=A(I-H)^{-1}-A=A(I-H)^{-1}H \geq 0$, since $H \geq 0$ and
$\rho(H)<1$ by Theorem \ref{pereathm}.
\end{proof}

We obtain the following corollary for a  square nonsingular matrix
$A$.

\begin{cor}
Let $(U_{k}, V_{k},E_{k})_{k=1}^{p}$ be a  weak regular
multisplitting of a monotone matrix $A \in {\R}^{n \times n}$. Then the splitting $A=B-C$ induced by $H$ is a
regular splitting if $A \geq 0$.
\end{cor}

Next theorem compares the spectral radii between a multisplitting
and a splitting of a real rectangular matrix $A$.

\begin{thm}\label{multi2}
Let $(U_{k}, V_{k},E_{k})_{k=1}^{p}$ be a proper weak regular
multisplitting of a semimonotone matrix $A \in {\R}^{m \times n}$
and $\underline{U},\overline{U} \in {\R}^{m \times n}$
 such that
  $$\overline{U}^{\dagger} \leq U_{k}^{\dagger}\leq \underline{U}^{\dagger}, ~\text{for each}~ k=1,2,\ldots,p.$$
   and $R(E_{k})\subseteq R(A^{T}), ~ \text{for each}~
 k=1,2,\ldots,p$.\\

$(i)$ If $A=\overline{U}-\overline{V}$ is a proper regular splitting
and row sums of $\overline{U}^{\dagger}$ are positive, then
$$\rho(H)\leq \rho(\overline{U}^{\dagger}\overline{V}).$$

 $(ii)$ If
$A=\underline{U}-\underline{V}$ is a proper regular splitting,
 then $$\rho(\underline{U}^{\dagger}\underline{V})\leq \rho(H).$$
\end{thm}

\begin{proof}
$(i)$ Let $\widetilde{U}_{1}=B,~
\widetilde{U}_{2}=\overline{U}~\text{and}~
\widetilde{V}_{2}=\overline{V}$. Then
$\widetilde{U}_{1}^{\dagger}\widetilde{V}_{1}=B^{\dagger}(B-A)=B^{\dagger}B-B^{\dagger}A=A^{\dagger}A-(I-H)A^{\dagger}A
=HA^{\dagger}A=H \geq 0.$ The condition $U_{k}^{\dagger} \geq
\overline{U}^{\dagger}$ implies $\widetilde{U}_{1}^{\dagger} \geq
\widetilde{U_{2}}^{\dagger},~\widetilde{V_{2}} \geq 0$. By Theorem
\ref{d4} $(iii)$, we then have $\rho(H) \leq
\rho(\overline{U}^{\dagger}\overline{V}).$\\

$(ii)$  Define
 $\widetilde{U}_{1}=\underline{U},~\widetilde{V}_{1}=\underline{V}~\text{and}~\widetilde{U}_{2}=B,$
 and on applying Theorem \ref{d4} $(ii)$, we obtain  $\rho(\underline{U}^{\dagger}\underline{V})\leq \rho(H)$.
\end{proof}

For a square nonsingular matrix $A$, the above result reduces to the following
corollary.

\begin{cor}
Let $(U_{k}, V_{k},E_{k})_{k=1}^{p}$ be a  weak regular
multisplitting of a monotone matrix $A \in {\R}^{n \times n}$ and
$\underline{U},\overline{U} \in {\R}^{n \times n}$
 such that $$\overline{U}^{-1} \leq U_{k}^{-1}\leq \underline{U}^{-1}, ~\text{for each}~ k=1,2,\ldots,p.$$

$(i)$ If $A=\overline{U}-\overline{V}$ is a regular splitting, then
$$\rho(H)\leq \rho(\overline{U}^{-1}\overline{V}).$$

$(ii)$ If $A=\underline{U}-\underline{V}$ is a  regular splitting,
 then $$\rho(\underline{U}^{-1}\underline{V})\leq \rho(H).$$
\end{cor}

The spectral radii of iteration matrices of two proper weak regular
multisplittings of the same coefficient matrix $A$ is compared
below.

\begin{thm}\label{multi4}
Let  $(U_{k}^{(i)}, V_{k}^{(i)},E_{k})_{k=1}^{p},~i=1,2$, be two
proper weak regular multisplittings of a non-negative semimonotone matrix
$A \in {\R}^{m \times n}$ such that $R(E_{k})\subseteq
R(A^{T}),~\text{for each}~ k=1,2,\ldots,p$. If $V_{k}^{(2)} \geq
V_{k}^{(1)}, ~\text{for each}~ k=1,2,\ldots,p,$
 then $$\rho(H_1)\leq \rho(H_2)<1,$$
where
$H_{i}=\displaystyle\sum_{k=1}^{p}E_{k}[U_{k}^{(i)}]^{\dagger}V_{k}^{(i)},~\text{for
each}~i=1,2.$
\end{thm}

\begin{proof}
From $V_{k}^{(2)} \geq V_{k}^{(1)}, ~\text{for each}~
~k=1,2,\ldots,p,$ we obtain $$U_{k}^{(2)} \geq U_{k}^{(1)},~\text{for each}~ k=1,2,\ldots,p.$$ Since
$R(U_{k}^{(1)})=R(U_{k}^{(2)})$ and
$N(U_{k}^{(1)})=N(U_{k}^{(2)})$ by Lemma \ref{lemmad4}, it
follows that $$[U_{k}^{(1)}]^{\dagger} \geq
[U_{k}^{(2)}]^{\dagger},~\text{for each}~  k=1,2,\ldots,p.$$
Consequently, $$\displaystyle
\sum_{k=1}^{p}E_{k}[U_{k}^{(1)}]^{\dagger} \geq
\sum_{k=1}^{p}E_{k}[U_{k}^{(2)}]^{\dagger},$$ i.e.,
$$B_{1}^{\dagger} \geq B_{2}^{\dagger}.$$ By
Theorem \ref{prms},  the splittings $A=B_{1}-C_{1}=B_{2}-C_{2}$
induced by $H_{1}$ and $H_{2}$ are proper regular splittings. Hence,
by Theorem \ref{calcolo}, we obtain  $\rho(H_{1}) \leq
\rho(H_{2})<1.$
\end{proof}

We have the following corollary.

\begin{cor}
Let $(U_{k}^{(i)}, V_{k}^{(i)}, E_{k})_{k=1}^{p},~i=1,2$,
be two weak regular multisplittings of a non-negative monotone matrix $A
\in {\R}^{n \times n}$.
 If $V_{k}^{(2)} \geq V_{k}^{(1)}, ~\text{for each}~ k=1,2,\ldots,p,$
 then $$\rho(H_1)\leq \rho(H_2)<1,$$
where
$H_{i}=\displaystyle\sum_{k=1}^{p}E_{k}[U_{k}^{(i)}]^{-1}V_{k}^{(i)},~\text{for
each}~i=1,2.$
\end{cor}

\begin{remark}
Theorem \ref{prms} and  Theorem \ref{multi4}  are also true if we assume $G^{\dagger} \geq 0$ instead of $A \geq 0$.
\end{remark}

Next result compares the spectral radii of iteration matrices of two
proper weak regular multisplittings of the same coefficient matrix
$A$.

\begin{thm}\label{multi6}
Let $(U_{k}^{(i)}, V_{k}^{(i)},E_{k})_{k=1}^{p},~ i=1,2,$ be two
proper weak regular multisplittings of a non-negative semimonotone
matrix $A \in {\R}^{m \times n}$ such that $R(E_{k})\subseteq
R(A^{T}), ~\text{for each}~ k=1,2,\ldots,p$. If
$[U_{k}^{(1)}]^{\dagger} \geq [U_{k}^{(2)}]^{\dagger}, ~\text{for
each}~ k=1,2,\ldots,p,$ then $$\rho(H_{1}) \leq \rho(H_{2})<1.$$
\end{thm}

\begin{proof} By Theorem \ref{prms}, the splittings $A=B_{1}-C_{1}=B_{2}-C_{2}$
 induced by $H_{1}=\displaystyle\sum_{k=1}^{p}E_{k}[U_{k}^{(1)}]^{\dagger}V_{k}^{(1)}$ and
 $H_{2}=\displaystyle\sum_{k=1}^{p}E_{k}[U_{k}^{(2)}]^{\dagger}V_{k}^{(2)}$ are  proper regular splittings.
From $$[U_{k}^{(1)}]^{\dagger} \geq [U_{k}^{(2)}]^{\dagger},
~\text{for each}~ k=1,2,\ldots,p,$$  we have
$$\displaystyle\sum_{k=1}^{p}E_{k} [U_{k}^{(1)}]^{\dagger}
\geq \sum_{k=1}^{p}E_{k}[U_{k}^{(2)}]^{\dagger},~\text{for each}~
k=1,2,\ldots p,$$
 i.e., $$B_{1}^{\dagger} \geq B_{2}^{\dagger}.$$ Hence, by Theorem \ref{calcolo},
  we obtain $\rho(H_{1}) \leq \rho(H_{2})<1.$

\end{proof}

  The following corollary  follows immediately from the above result
  when a square nonsingular system  of linear equations is considered.

\begin{cor}
Let $(U_{k}^{(i)}, V_{k}^{(i)},E_{k})_{k=1}^{p},~
i=1,2,$ be two weak regular multisplittings  of a non-negative
monotone matrix $A \in {\R}^{n \times n}$.
 If $[U_{k}^{(1)}]^{-1} \geq
[U_{k}^{(2)}]^{-1}, ~\text{for each}~ k=1,2,\ldots,p,$ then
$$\rho(H_{1}) \leq \rho(H_{2})<1.$$
\end{cor}

\section{Conclusions}
The notion of proper multisplittings proposed by Climent and Perea
\cite{c3}  is  interesting notion for solving the singular and
 rectangular linear systems. It
allows us to get the solution in a parallel perspective. The
convergence of the iterative method (\ref{climenteq}) is then ensured by the same
authors. In this work, a few comparison results are shown. Apart
from these, the type of the induced splitting induced by the
iteration matrix formed by proper multisplittings is guaranteed
under some hypotheses. This result also makes a contribution to the
convergence of the induced splitting (Theorem 5.5). The results
discussed in Section  3 and 4 compare the spectral radii of the
iteration matrices formed by different types of matrix splittings.

\vspace{1cm}

\noindent {\small {\bf Acknowledgments.}\\
We wish to thank  K. C.  Sivakumar for his valuable comments on an earlier version of
the manuscript. The second author acknowledges the support provided by Chhattisgarh Council of Science
and Technology, Chhattisgarh, India under the grant number:
2545/CCOST/MRP/2016.

\end{document}